\documentclass[11pt,oneside]{article}

\usepackage[utf8]{inputenc} 
\usepackage[T1]{fontenc}    
\usepackage[english]{babel}

\usepackage{amsfonts}
\usepackage{nicefrac}       
\usepackage{microtype}      
\usepackage{lmodern}
\usepackage{amssymb,amsmath,amsthm}

\usepackage{stmaryrd}
\SetSymbolFont{stmry}{bold}{U}{stmry}{m}{n}

\usepackage{bbm,bm}
\usepackage{latexsym}
\usepackage{xcolor}

\usepackage{enumerate}
\usepackage{verbatim}
\usepackage{booktabs}       

\usepackage{url}            
\usepackage[colorlinks=true]{hyperref}
\usepackage[numbers,sort&compress,square,comma]{natbib}
\usepackage{cleveref}

\usepackage{parskip}
\usepackage{geometry}
\usepackage[short]{optidef}
\usepackage{algorithmic,algorithm}

\usepackage{thmtools}

\declaretheorem{theorem}

\declaretheorem{lemma}

\declaretheorem{proposition}

\declaretheorem{fact}

\declaretheoremstyle[qed=$\square$]{definitionwithend}

\declaretheorem[style=definitionwithend]{assumption}

\declaretheorem[style=definitionwithend]{remark}

\crefname{assumption}{Assumption}{Assumptions}
\crefname{setting}{Setting}{Setting}
\crefname{conjecture}{Conjecture}{Conjectures}
\crefname{fact}{Fact}{Facts}

\newcommand{\by}{\times}
 
\newcommand{\norm}[1]{\ensuremath{\left\lVert #1 \right\rVert}}
\newcommand{\ip}[1]{\ensuremath{\left\langle #1 \right\rangle}}
\newcommand{\grad}{\ensuremath{\nabla}}

\newcommand{\set}[1]{\left\{#1\right\}}

\def\R{{\mathbb{R}}}

\newcommand{\framedheader}[3]{
  \framebox[\textwidth]{
    \vbox{
      \vspace{2mm}
      \hbox to \textwidth {\hspace{1em}\today \hfill #1\hspace{1em}}
      \vspace{4mm}
      \hbox to \textwidth {\hfill \Large{#2} \hfill}
      \vspace{2mm}
    }
  }
  \vspace*{4mm}
}
\input{alex-dots}
\usepackage{soul}

\newcommand{\vr}{\varrho}
\newcommand{\bbracket}[1]{\ensuremath{\left\llbracket #1 \right\rrbracket}}

\newcommand{\hleft}{\mathfrak{h}_{\textup{left}}}
\newcommand{\hright}{\mathfrak{h}_\textup{right}}

\title{Accelerated Objective Gap and Gradient Norm\\ Convergence for Gradient Descent via Long Steps\footnote{This manuscript greatly simplifies and improves the analysis of stepsize patterns that first appeared in an unpublished technical report by the same authors~\cite{TechnicalLongSteps}.}}
\date{\today}
\author{
    Benjamin Grimmer\footnote{Johns Hopkins University, Department of Applied Mathematics and Statistics, \texttt{grimmer@jhu.edu}}
    \and
    Kevin Shu\footnote{Georgia Institute of Technology, School of Mathematics, \texttt{kshu8@gatech.edu}}
    \and
    Alex L.\ Wang\footnote{Purdue University, Daniels School of Business, \texttt{wang5984@purdue.edu}}
}

\begin{document}

\maketitle

\begin{abstract}
This work considers gradient descent for $L$-smooth convex optimization with stepsizes larger than the classic regime where descent can be ensured.
The stepsize schedules considered are similar to but differ slightly from the recent silver stepsizes of Altschuler and Parrilo.
For one of our stepsize sequences, we prove a $O(1/N^{1.2716\dots})$ convergence rate in terms of objective gap decrease and for the other, we show the same rate of decrease for squared-gradient-norm decrease. This first result improves on the recent result of Altschuler and Parrilo by a constant factor, while the second results improve on the exponent of the prior best squared-gradient-norm convergence guarantee of $O(1/N)$.

\end{abstract}

\section{Introduction}
Gradient descent is one of the most classic optimization methods---minimizing a function by repeatedly updating a candidate solution in the steepest descent direction. Here, we consider unconstrained minimization of an $L$-smooth convex function $f\colon \mathbb{R}^n \rightarrow \mathbb{R}$ via $N$ steps of gradient descent
\begin{equation}
	x_{i+1} = x_i - \frac{h_i}{L}\nabla f(x_i) \qquad \forall i=0\dots N-1  \label{eq:GD}
\end{equation}
given an initialization $x_0\in\mathbb{R}^n$ and stepsize schedule $h = (h_0,h_1,\dots, h_{N-1})$ normalized by $1/L$. In this short work, we provide selections of stepsizes $h_i$ yielding best-known convergence guarantees in terms of the terminal iterate's objective gap $f(x_N) - f(x_\star)$ and squared gradient norm $\|\nabla f(x_N)\|^2$.

Classically, the literature has primarily considered stepsizes $h_i \in (0,2)$ as one can guarantee a monotone decrease in the objective value at each iteration.
We will refer to the regime where $h_i\in(0,2)$ as the ``short stepsize'' regime.
The works~\cite{drori2012PerformanceOF,taylor2017interpolation,taylor2017smooth,Teboulle2022}, enabled by considering related performance estimation problems, have established tight descriptions of the convergence of gradient descent under a range of stepsize policies in this regime. When convergence is measured by $f(x_N) - f(x_\star)$ or by $\|\nabla f(x_N)\|^2$, the best convergence rates are of the form $O(L\|x_0-x_\star\|^2/N)$ or $O(L(f(x_0)-f(x_\star))/N)$, respectively.

Recently, there has been interest in developing theory beyond the short stepsize regime. Although using ``long steps'' means the objective gap may no longer monotonically decrease, it also provides an opportunity to improve the quality of the final iterate. The theses of Daccache~\cite{Daccache2019} and Eloi~\cite{Eloi2022} provided complete characterizations of gradient descent's performance for smooth convex optimization when $N\leq 3$.
A numerical search using a branch-and-bound methodology was conducted by~\cite{gupta2023branch} to identify optimal stepsizes up to $N=50$.
For strongly convex optimization, Altschuler's thesis~\cite[Chapter 8]{altschuler2018greed} derived the optimal stepsizes when $N\leq 3$ for contracting either the distance to optimal or the gradient's norm. Inductively applying these patterns produced the first convergence guarantees for a long step method. More recently, Grimmer~\cite{Grimmer2023-long} provided an inductive proof technique for non-strongly convex problems based on ``straightforward'' stepsize patterns, showing increasing improvements in the constants of convergence rates from repeating a pattern of up to $127$ stepsizes.

Concurrently, two lines of work carried these long step ideas forward, designing sequences of stepsizes which improve the classic $O(1/N)$ convergence rates of gradient descent with short steps by a polynomial factor. Expanding their prior work~\cite{altschuler2018greed}, Altschuler and Parrilo proposed the ``silver'' stepsize patterns for strongly convex~\cite{altschuler2023accelerationPartI} and then for convex optimization~\cite{altschuler2023accelerationPartII}, based on the silver ratio $\vr:=1+\sqrt{2}$. For smooth convex optimization, they derive an objective gap convergence rate of $O(1/N^{\log_2(\vr)})$ where $\log_2(\vr)\approx 1.2716$. At the same time, this work's authors designed arbitrarily long sequences of straightforward stepsizes, facilitating the proof of a weaker $O(1/N^{1.0564})$ convergence rate for the objective gap, presented in the technical report~\cite{TechnicalLongSteps}.

This paper considers the same stepsize sequences that we proposed in our prior unpublished technical report~\cite{TechnicalLongSteps} but provides much shorter and stronger analyses of these sequences. We will show that our stepsize sequences achieve an $O(1/N^{\log_2(\vr)})$ convergence rate in objective gap and squared gradient norm. Our guarantees on objective gap convergence improve upon our previous $O(1/N^{1.0564})$ rate in exponent and upon the convergence rate of \cite{altschuler2023accelerationPartII} by constant factors. Our results on squared gradient norm convergence appear to be entirely new, and the first results improving upon gradient descent's classic short stepsize $O(1/N)$ rates.
Our improved analysis and its presentation is inspired directly by the recent work of Altschuler and Parrilo~(see \cref{rem:hist}).

\paragraph{Our Contributions.} Our results utilize a sequence of stepsizes $\mathfrak{h}^{(k)}_\mathtt{left} \in \R^N$, where $N=2^k-1$. We will also define $\mathfrak{h}^{(k)}_\mathtt{right}$ to be the sequence obtained by reversing the entries of $\mathfrak{h}^{(k)}_\mathtt{left}$.
Our convergence guarantees rely on the following quantity definable either in terms of sums or products of our proposed stepsizes
$$ r_k := 2\sum_{i=0}^{N-1}(\mathfrak{h}^{(k)}_\mathtt{left})_i +1 = \prod_{i=0}^{N-1} (1 - (\mathfrak{h}^{(k)}_\mathtt{left})_i)^{-2} \simeq \left((\vr - 1)(1 + \vr^{-1/2})\right)N^{\log_2(\vr)} \ . $$
Here, the $\simeq$ sign denotes equality up to lower-order terms.

We defer exact definitions of $\hleft^{(k)}$ and $\hright^{(k)}$ to \Cref{sec:constructions} and first state our guarantees.

\begin{theorem}\label{thm:h_left}
	For any $L$-smooth convex $f$ with a minimizer $x_\star$ and $k\geq 0$, gradient descent~\eqref{eq:GD} with stepsizes $h=\mathfrak{h}^{(k)}_\mathtt{left}$ after $N=2^k-1$ steps has
	$$ \frac{1}{2}\|\nabla f(x_{N}) \|^2 \leq \frac{L(f(x_0) - f(x_\star))}{r_k} \simeq \frac{1}{(\vr - 1)(1 + \vr^{-1/2})}\cdot\frac{L(f(x_0) - f(x_\star))}{N^{\log_2(\vr)}} \ . $$
\end{theorem}
\begin{theorem}\label{thm:h_right}
	For any $L$-smooth convex $f$ with a minimizer $x_\star$ and $k\geq 0$, gradient descent~\eqref{eq:GD} with stepsizes $h=\mathfrak{h}^{(k)}_\mathtt{right}$ after $N=2^k-1$ steps has
	$$ f(x_N)-f(x_\star) \leq \frac{\frac{L}{2}\|x_0 - x_\star\|^2}{r_k}  \simeq \frac{1}{(\vr - 1)(1 + \vr^{-1/2})}\cdot\frac{\frac{L}{2}\|x_0 - x_\star\|^2}{ N^{\log_2(\vr)}} \ . $$
\end{theorem}
\begin{theorem}\label{thm:h_right_h_left}
	For any $L$-smooth convex $f$ with a minimizer $x_\star$ and $k\geq 0$, gradient descent~\eqref{eq:GD} with stepsizes $h=[\mathfrak{h}^{(k)}_\mathtt{right}, \mathfrak{h}^{(k)}_\mathtt{left}]$ after $N=2^{k+1}-2$ steps has
	$$ \|\nabla f(x_{N}) \| \leq \frac{L\|x_0 - x_\star\|}{r_k}  \simeq \frac{\vr}{(\vr -1)(1+\vr^{-1/2})}\frac{L\|x_0 - x_\star\|}{N^{\log_2(\vr)}} \ . $$
\end{theorem}
The first two theorems are proven in Sections~\ref{sec:h_left} and~\ref{sec:h_right}. The third result is an immediate consequence of the first two as
$$ \frac{1}{2}\|\nabla f(x_{2^{k+1}-2}) \|^2 \leq \frac{L(f(x_{2^k -1}) - f(x_\star))}{r_k} \leq \frac{\frac{L^2}{2}\|x_0 - x_\star\|^2}{r_k^2} $$
and $(N/2)^{-\log_2(\vr)} = \vr N^{-\log_2(\vr)}$.

\Cref{tab:numeric_comp} compares the convergence guarantee in \Cref{thm:h_right} with that of \cite{altschuler2023accelerationPartII} for the silver stepsize schedule, and \cite{gupta2023branch} for the stepsize schedules found via branch-and-bound for $N = 1,3,7,15,31$.

\begin{table}[t]
    \centering
    \begin{tabular}{cccc}\toprule
    $N$ & \Cref{thm:h_right} & \cite[Theorem 1.1]{altschuler2023accelerationPartII} & \cite[Section 6.1]{gupta2023branch} \\ \hline
    $1$ & 0.125 & 0.182 & 0.125\\
    $3$ & 0.0429 & 0.0798 & 0.0429\\
    $7$ & 0.0164 & 0.0344& 0.0163\\
    $15$ & 0.00654& 0.0145 & 0.00659\\
    $31$ & 0.00266& 0.00606 & 0.00272\\
    \bottomrule
    \end{tabular}
    \caption{Numeric comparison of the convergence guarantees of \Cref{thm:h_right}, \cite[Theorem 1.1]{altschuler2023accelerationPartII}, and \cite[Section 6.1]{gupta2023branch} for $N=1,3,7,15,31$ steps of gradient descent. Each table entry corresponds to an upper bound on $f(x_N)- f(x^\star)$ assuming that $\norm{x_0-x^\star}\leq 1$ and $f$ is a $1$-smooth convex function.}
    \label{tab:numeric_comp}
\end{table}

\begin{remark}[On Optimality of Rates]
	Theorem~\ref{thm:h_right} improves on the convergence rate of the silver stepsize schedule~\cite{altschuler2023accelerationPartII} asymptotically by a constant factor of $(\vr - 1)(1+\vr^{-1/2})\approx 2.32439$. Therein, Altschuler and Parrilo claim that a forthcoming paper will show no gradient descent stepsize scheme can do better than $O(1/N^{\log_2(\vr)})$ in objective gap convergence.
	
	Theorems~\ref{thm:h_left} and~\ref{thm:h_right} show that applying the same stepsizes in reverse order mirrors the convergence rate between gradient norm and objective gap. This symmetry appears closely related to the H-duality theory of~\cite{kim2023timereversed}. Therein, they show for methods with a certain type of inductive proof, reversing the order of steps (and momentum) provides a similar conversion between gradient norm and objective gap convergence. If such a symmetry is fundamental, conjecturing $O(1/N^{\log_2(\vr)})$ is also the optimal rate for gradient norm convergence is well motivated.
	
	Note although $O(1/N^{\log_2(\vr)})$ is conjectured to be the optimal rate for gradient descent among all possible stepsize selections, it is not optimal among all gradient methods. Nesterov's accelerated method~\cite{Nesterov1983} attains an objective gap rate of $O(1/N^2)$, matching the $O(1/N^2)$ lower bound for gradient methods of~\cite{Yudin1982}. Similarly, $O(1/N^2)$ is the optimal convergence rate for reducing gradient norm, attained by the method of~\cite{KimFessler2021}.
\end{remark}
\begin{remark}[On Tightness of our Analysis]
	All three of these results are tight for the step size sequences considered here, i.e., there exists an $L$-smooth convex function attaining our proven bounds on convergence. As a result, any further improvement in convergence rates (even by constants) will require a modification of the stepsize schedule. 
    We provide two (univariate) worst-case functions for each result:
    First, consider the quadratic function $f(x) = \frac{L}{2}x^2$. For any $x_0$ and stepsize sequence $h$, gradient descent's terminal iterate is $\prod_{i=0}^{N-1} (1-h_i) x_0$. For any $L$ and $x_0$, a simple calculation verifies that Theorems~\ref{thm:h_left}-\ref{thm:h_right_h_left} are met with equality.
	Next, consider the family of Huber functions of the form $\phi_{\eta}(x) = \frac{L}{2}x^2$ if $|x|\leq \eta$ and $L\eta|x| - \frac{L\eta^2}{2}$ otherwise.
    Note that $x_\star = 0$ and $f_\star = 0$ for any Huber function of this form.
    Direct calculations shows that for any $x_0$, \Cref{thm:h_left} is tight when $\eta =2|x_0|/(1+r_k)$ and
    Theorems~\ref{thm:h_right} and~\ref{thm:h_right_h_left} are tight when $\eta=|x_0|/r_k$.
	A simple consequence of Huber functions attaining the worst-case performance, observed in~\cite{luner2024averaging}, is that reporting an averaging of the iterates $\sum \sigma_i x_i$ instead of $x_N$ strictly worsens the worst-case performance of gradient descent with our proposed stepsizes. 
\end{remark}
\begin{remark}[On Faster Linear Convergence]
	For $\mu$-strongly convex optimization, gradient descent with $h_i=1$ contracts the objective gap and the distance to $x_\star$ every iteration.
	Similarly, our Theorems~\ref{thm:h_left}-\ref{thm:h_right_h_left} all yield contractions and hence, if repeatedly applied, accelerated linear convergence rates under an additional assumption implied by $\mu$-strong convexity. If $\frac{1}{2}\|\nabla f(x)\|^2 \geq \mu(f(x) - f(x_\star))$ holds for all $x\in\mathbb{R}^n$, then Theorem~\ref{thm:h_left} ensures a contraction of $f(x_N) - f(x_\star) \leq \frac{L}{\mu r_k} (f(x_0) - f(x_\star))$.
    We can set $r_k \geq 2L/\mu$ by taking $k\simeq \log_\vr(L/\mu)$, so that $f(x_N)-f(x_\star)\leq \frac{1}{2}(f(x_0)-f(x_\star))$. Consequently, repeatedly applying the sequence of stepsizes $\mathfrak{h}^{(k)}_\mathtt{left}$ yields an accelerated linear rate of $O\left((L/\mu)^{\log_\vr 2}\log(1/\epsilon)\right)$ gradient descent steps to achieve a suboptimality of $\epsilon(f(x_0)-f(x_\star))$. Similarly, if $f(x)-f(x_\star) \geq \frac{\mu}{2}\|x-x_\star\|^2$ or $\|\nabla f(x)\| \geq \mu\|x-x_\star\|$ holds for all $x\in\mathbb{R}^n$, Theorems~\ref{thm:h_right} and~\ref{thm:h_right_h_left} provide accelerated linear convergence rates on the distance to $x_\star$ when applied, respectively.
\end{remark}

\begin{remark}[Historical notes]
    \label{rem:hist}
The stepsize sequences in the current paper first appeared in our unpublished technical report~\cite{TechnicalLongSteps}. In that paper, we considered the sequence $\frac{1}{2}[\mathfrak{h}^{(k)}_\mathtt{left},r_k+1,\mathfrak{h}^{(k)}_\mathtt{right}]$ and showed via the straightforward machinery of \cite{Grimmer2023-long} that one could use these sequences to achieve $O(1/N^{1.056})$ convergence rates.
The proof of straightforwardness required a tedious calculation verifying that an exponentially-sized matrix $\lambda^{(k)}\in\R^{(2^{k+1}+1) \times (2^{k+1}+1)}$, which was constructed entry by entry, was a valid certificate for straightforwardness.
Concurrent to~\cite{TechnicalLongSteps}, Altschuler and Parrilo~\cite{altschuler2023accelerationPartII,altschuler2023accelerationPartI} presented the silver stepsize sequence and showed that it could achieve an $O(1/N^{\log_2\vr})$ convergence rate. Although their proof also relied on building and verifying an exponentially-sized certificate, it had much more streamlined inductive nature due to the recursive definition of their certificates. They called this technique \emph{recursive gluing}.

The convergence rates in the present paper for $\mathfrak{h}^{(k)}_\mathtt{left}$ and $\mathfrak{h}^{(k)}_\mathtt{right}$ will be proved using certificates that are \emph{submatrices} of our previous certificate $\lambda^{(k)}$.
Although the matrix $\lambda^{(k)}$ (and its submatrices) have not changed, we will give a new \emph{recursive} definition of these certificates that allow us to use the recursive gluing technique of~\cite{altschuler2023accelerationPartII,altschuler2023accelerationPartI} to similarly shorten our previously tedious calculations.
\end{remark}


\section{Constructing Stepsize Schedules $\hleft^{(k)}$ and $\hright^{(k)}$}
\label{sec:constructions}

Begin by defining $\beta_k\coloneqq 1 + \vr^{k-1}$ for all $k\geq 0$, where $\vr = 1+\sqrt{2}$ is the previously introduced \emph{silver ratio} which arose in the stepsize schedules of~\cite{gupta2023branch,altschuler2018greed,altschuler2023accelerationPartI,altschuler2023accelerationPartII,TechnicalLongSteps}.
For $k\geq 1$, we will define the $k$th \emph{silver stepsize schedule} $\pi^{(k)} \in\R^{2^k - 1}$ as
\begin{gather}
	\pi^{(1)} \coloneqq [\sqrt{2}]
	\qquad
	\text{and}\qquad
	\pi^{(k+1)} \coloneqq [\pi^{(k)},\beta_k, \pi^{(k)}]. \label{eq:silver_def}
\end{gather}
We will $0$-index $\pi^{(k)}$.
Note that with this indexing convention, we have that
$\pi^{(k)}_i = \beta_{\nu(i+1)}$,
where $\nu(m)$ denotes the exponent on $2$ in the prime factorization of $m$.

The silver stepsize schedule was shown~\cite{altschuler2023accelerationPartII} to achieve the following guarantee on $L$-smooth convex functions: Let $N = 2^k - 1$, then gradient descent with stepsizes $h=\pi^{(k)}$ has
\begin{equation}
	f(x_N) - f^* \leq \frac{L\norm{x_0 - x_\star}^2}{\sqrt{1+4\vr^{2k}-3}} \simeq \frac{L\norm{x_0 - x_\star}^2}{2N^{\log_2(\vr)}}. \label{eq:silver_rate}
\end{equation}

Although this stepsize schedule achieves the conjectured optimal rate of $O(1/N^{\log_2(\vr)})$, it is not optimal in terms of constants. For example, the branch-and-bound optimization over all stepsizes of~\cite{gupta2023branch} found the optimal one step is $[3/2]$ and the optimal three steps are $[\sqrt{2}, 1+\sqrt{2}, 3/2]$, in contrast to the silver patterns of $[\sqrt{2}]$ and $[\sqrt{2},2,\sqrt{2}]$.

We consider a similar construction to the silver stepsize schedule, modified to produce the optimal $N=1$ and $N=3$ stepsizes. Given that the optimal three steps $[\sqrt{2}, 1+\sqrt{2}, 3/2]$ are no longer mirror symmetric, placing a larger weight on the right, we denote the corresponding stepsize schedule $\hright^{(k)}\in\R^{2^{k} - 1}$. As with $\pi^{(k)}$, we $0$-index $\hright^{(k)}$ and its mirror $\hleft^{(k)}$.

Our ``right-side heavy'' stepsizes are recursively defined, utilizing the silver stepsizes as
\begin{gather}
	\hright^{(1)} \coloneqq [3/2]
	\qquad
	\text{and}\qquad
	\hright^{(k+1)} \coloneqq [\pi^{(k)},\alpha_k, \hright^{(k)}]. \label{eq:h_right_def}
\end{gather}
where $\alpha_k$ is the unique solution to
\begin{align}
	\label{eq:defining_alpha}
	\begin{cases}
		\alpha_k\geq 1\\
		1 + 2\sum_{i=0}^{2^{k+1} - 2} (\hright^{(k+1)})_i = \prod_{i=0}^{2^{k+1} - 2} \left((\hright^{(k+1)})_i - 1\right)^{-2}
	\end{cases}.
\end{align}

\begin{remark}
\label{rem:dasgupta}
This construction matches the stepsizes found via numeric search in~\cite{gupta2023branch} for $\hright^{(1)}$, $\hright^{(2)}$.
By solving the associated performance estimation problem, we numerically find our construction $\hright^{(3)}$ has a worse worst-case convergence guarantee than the stepsize sequence found in~\cite{gupta2023branch} of length 7.
Interestingly, our constructions $\hright^{(4)}$, and $\hright^{(5)}$ outperform the stepsize sequences numerically found by~\cite{gupta2023branch} of length 15 and 31 respectively.
\end{remark}

We denote the common value that the definition of $\alpha_k$ ensures by
$$r_k \coloneqq 1 + 2\sum_{i=0}^{2^k - 2} (\hright^{(k)})_i = \prod_{i=0}^{2^k - 2} \left((\hright^{(k)})_i - 1\right)^{-2}.$$
For example, $r_1=4$. An equivalent \emph{explicit} recurrence for generating $\alpha_k$ and $r_k$ is stated in Lemma~\ref{lem:explicit_alpha_mu}. Enforcing that these two values are equal is motivated by the conjectured optimal choice of a constant stepsize being exactly the one that balances these two terms; this was first conjectured in~\cite{drori2012PerformanceOF} and further studied and motivated by~\cite{Taylor2015SmoothSC}. Setting these two quantities equal ensures that the worst case convergence rate over quadratic functions and Huber functions are equal.

We define $\hleft^{(k)}$ to be the steps in $\hright^{(k)}$ in reverse order. Explicitly:
\begin{gather}
	\hleft^{(1)} \coloneqq [3/2]
	\qquad
	\text{and}\qquad
	\hleft^{(k+1)} \coloneqq [\hleft^{(k)},\alpha_k, \pi^{(k)}]. \label{eq:h_left_def}
\end{gather}

Useful identities relating $\alpha_k$, $\beta_k$, and $r_k$ can be found in \Cref{sec:useful_identitites}.

\subsection{Proof Strategy for Theorems~\ref{thm:h_left} and~\ref{thm:h_right}}
Without loss of generality, we can assume the smoothness constant $L$ is equal to one by considering the objective $\frac{1}{L}f$ instead.
We will repeatedly use the following assumption, so we abbreviate it:
\begin{assumption}
	\label{as:first_order_info}
	Suppose $f_0, \dots, f_N, f_\star$, $x_0, \dots, x_N, x_\star$, and $g_0, \dots, g_N, g_\star$ is the set of first-order information generated by gradient descent with stepsizes $h=(h_0,\dots, h_{N-1})$ starting at $x_0$ on a $1$-smooth convex function $f$. That is, $f(x_i) = f_i$ and $\grad f(x_i) = g_i$ with $g_\star=0$.
\end{assumption}
Our proofs of \Cref{thm:h_left,thm:h_right} will use the following classic result~\cite[Theorem 2.1.5]{nesterov-2018-textbook}:
\begin{fact}
	\label{fact:coercivity}
	Suppose $f:\R^n\to\R$ is a $1$-smooth convex function and $x,y\in\R^n$. Then,
	\begin{align*}
		Q_{x,y}\coloneqq f(x) - f(y) - \ip{\grad f(y), x - y} - \frac{1}{2}\norm{\grad f(x) - \grad f(y)}^2
	\end{align*}
	is nonnegative.
\end{fact}

When the set $\set{x_\star, x_0, x_1,\dots,x_n}$ is clear from context, we will abbreviate $Q_{i,j}\coloneqq Q_{x_i, x_j}$.

\begin{remark}
	\label{rem:interpolation}
	There is a strong converse to \Cref{fact:coercivity} given by the interpolation theorem of~\cite{taylor2017interpolation}, which shows that a set of first-order information $\set{x_i, f_i, g_i}$ is interpolable via a $1$-smooth convex function, i.e., there exists a $1$-smooth convex function $f$ satisfying $f(x_i) = f_i$ and $\grad f(x_i) = g_i$, if and only $Q_{i,j}\geq 0$ for all pairs $i\neq j$. We will not require this converse in our proofs.
\end{remark}

\Cref{thm:h_left,thm:h_right} follow from similar proof strategies. The critical technique is to derive nontrivial inequalities by considering carefully selected nonnegative combinations of the $Q_{i,j}$ as $\sum_{i,j} \lambda_{i,j}Q_{i,j}\geq 0$ for any selection $\lambda\geq 0$ by \Cref{fact:coercivity}. This technical idea has been extensively explored in the performance estimation problem literature over the past decade.
In particular, our convergence theory relies on the following two nontrivial inequalities.
\begin{proposition}
	\label{thm:grad_norm_decrease}
	Let $k \geq 1$ and invoke \Cref{as:first_order_info} with $h = \hleft^{(k)}$. Then, there exist nonnegative multipliers $A^{(k)}\in\R^{2^k \times 2^k}$ such that 
	\begin{align*}
		f_0 - f_{2^{k}-1} - \frac{r_k - 1}{2} \norm{g_{2^{k}-1}}^2 = \sum_{i,j=0}^{2^k-1} A^{(k)}_{i,j}Q_{i,j} \geq 0.
	\end{align*}
\end{proposition}
\begin{proposition}
	\label{thm:fn_val_decrease}
	Let $k \geq 1$ and invoke \Cref{as:first_order_info} with $h = \hright^{(k)}$. Then, there exist nonnegative multipliers $D^{(k)}\in\R^{2^k\by 2^k}$ and a nonnegative vector $c_k\in\R^{2^k}$ such that
	\begin{align*}
		&\sum_{i=0}^{2^{k} - 1} \frac{(c_k)_i }{\sqrt{r_k}}\left(f_{i} - f_{2^{k} - 1} + \frac{1}{2}\norm{g_{i}}^2 + \ip{g_i, x_0 - x_i}\right) - \frac{1}{2}\norm{\sum_{i=0}^{2^k - 1}(c_k)_i g_i}^2 = \sum_{i,j=0}^{2^k-1} D^{(k)}_{i,j}Q_{i,j} \geq 0.
	\end{align*}
    Furthermore, $\sum_{i=0}^{2^k-1}(c_k)_i=\sqrt{r_k}$.
\end{proposition}
For both equalities claimed above, observe that the expression on the left is linear in $f_0,f_1,\dots$ and homogeneously quadratic in $g_0,g_1,\dots$. This holds too for each $Q_{i,j}$: when $i < j$, we have that $x_j = x_i - \sum_{\ell = i}^{j-1} h_\ell g_\ell$ so that 
\begin{align*}
	Q_{i,j} &\coloneqq f_i - f_j - \ip{g_j, x_i - x_j} - \frac{1}{2}\norm{g_i - g_j}^2\\
	&= f_i - f_j - \sum_{\ell=i}^{j-1}h_\ell\ip{g_j, g_\ell} - \frac{1}{2}\norm{g_i - g_j}^2.
\end{align*}
An analogous computation holds when $i > j$.
Thus, it will suffice to prove that the terms depending on $f_i$ match and that the terms depending on $g_i$ match. We will let $\bbracket{\cdot}_f$ denote all terms depending on $f_i$ in a given expression. Similarly, define $\bbracket{\cdot}_g$. For example, if $i<j$, then
\begin{align*}
	\bbracket{Q_{i,j}}_f = f_i - f_j\qquad\text{and}\qquad
	\bbracket{Q_{i,j}}_g = 
	-\sum_{\ell = i}^{j-1}h_\ell\ip{g_j, g_\ell} - \frac{1}{2}\norm{g_i - g_j}^2.
\end{align*}

The full proofs of these two propositions are deferred momentarily: Section~\ref{sec:silver} provides an intermediate construction and helpful lemma regarding silver stepsizes and then Sections~\ref{sec:h_left} and~\ref{sec:h_right} prove Propositions~\ref{thm:grad_norm_decrease} and~\ref{thm:fn_val_decrease} by inductively constructing their needed multipliers. A few symbolically intense simplifications, amounting to showing certain $4\by 4$ or $5\by 5$ matrices are identically zero, are deferred to the associated Mathematica~\cite{mathematica} notebook available at the Github
repository \url{https://github.com/bgrimmer/GDLongSteps2024}. The multipliers proving these key propositions were also present in our previous technical report~\cite{TechnicalLongSteps}. Namely, the certificates $A^{(k)}$ and $D^{(k)}$ correspond to the upper left and lower right blocks of the certificate, denoted $\lambda^{(k)}$, therein.

From these two propositions, Theorems~\ref{thm:h_left} and~\ref{thm:h_right} follow immediately.
\begin{proof}[Proof of Theorem~\ref{thm:h_left}]
	Noting $f_{2^{k}-1} - f_\star -\frac{1}{2}\|g_{2^k-1}\|^2 = Q_{2^k-1,\star} \geq 0$, \Cref{thm:grad_norm_decrease} implies
	$$ f_0 - f_{\star} - \frac{r_k}{2} \norm{g_{2^{k}-1}}^2 = \sum_{i,j=0}^{2^k-1} A^{(k)}_{i,j}Q_{i,j} + Q_{2^k-1,\star} \geq 0.$$
	Reorganizing terms gives the theorem's bound, and \Cref{lem:rk_growth} verifies the asymptotic constant.
\end{proof}
\begin{proof}[Proof of Theorem~\ref{thm:h_right}]
	Noting $f_{\star} - f_i -\ip{g_i,x_\star-x_i} -\frac{1}{2}\|g_{i}\|^2 = Q_{\star,i} \geq 0$, \Cref{thm:fn_val_decrease} implies
	\begin{align*}
		&\sum_{i=0}^{2^{k} - 1} \frac{(c_k)_i }{\sqrt{r_k}}\left(f_{\star} - f_{2^{k} - 1} + \ip{g_i, x_0 - x_\star}\right) - \frac{1}{2}\norm{\sum_{i=0}^{2^k - 1}(c_k)_i g_i}^2\\
		&= \sum_{i,j=0}^{2^k-1} D^{(k)}_{i,j}Q_{i,j} + \sum_{i=0}^{2^k-1} \frac{(c_k)_i }{\sqrt{r_k}} Q_{\star,i} \geq 0.
	\end{align*}
    By \Cref{lem:sum_ck}, we have that $\sum_{i=0}^{2^k-1}(c_k)_i=\sqrt{r_k}$, thus completing the square above and reorganizing terms, we conclude that
	$$ f_{2^{k} - 1} - f_\star \leq \frac{\frac{1}{2}\norm{x_0-x_\star}^2 - \frac{1}{2}\norm{y-x_\star}^2}{r_k}$$
	where $y = x_0 - \sqrt{r_k}\sum_{i=0}^{2^k - 1}(c_k)_i g_i$. Bounding $\norm{y-x_\star}\geq 0$ gives the theorem's bound, and \Cref{lem:rk_growth} verifies the asymptotic constant.
\end{proof}

\section{A Helpful Lemma on the Silver Stepsizes $\pi^{(k)}$}\label{sec:silver}
Define the following sequences of matrices $B^{(k)}\in \R^{2^k\times 2^k}$ inductively.
We will $0$-index both the rows and columns of $B^{(k)}$.
Define
\begin{gather*}
    B_1 \coloneqq \begin{pmatrix} 0 & 1 \\ \vr^{-1} & 0\end{pmatrix},\qquad\text{and}\\
B_{k+1} \coloneqq 
\begin{pmatrix}
  B_{k} &  \\
   & \vr^2 B_{k}
\end{pmatrix}
 + 
    \begin{pmatrix}
      0 \\
      & 0 & \pi^{(k)} & 1 \\
       & & 0 & \\
         &  \vr^{k-1} & \pi^{(k)} &  0
    \end{pmatrix}.
\end{gather*}
Here, in the second definition, the second matrix has support
\begin{align*}
    \set{(2^{k}-1,i):\, i\in [2^{k}-1, 2^{k+1} - 1]}
    \cup
    \set{(2^{k+1}-1, i):\, i\in [2^{k}-2, 2^{k+1}- 2]}.
\end{align*}

\begin{lemma}
    \label{prop:p_b}
Let $k \geq 1$ and invoke \Cref{as:first_order_info} with $h = \pi^{(k)}$. Then, $p_B^{(k)}\coloneqq \sum_{i,j}B_{i,j}Q_{i,j}$ equals
\begin{align*}
    p_B^{(k)}&= -\frac{1}{\vr} \sum_{i=0}^{2^k - 2} h_i \left(f_{2^k - 1} - f_i - \frac{1}{2}\norm{g_i}^2 - \ip{g_i, x_0 - x_i}\right)\\
    &\qquad\qquad - \frac{1}{2\vr}\norm{x_{2^k - 1} - x_0}^2 - \frac{\vr^{k-1}(\vr^k - 1)}{2}\norm{g_{2^k - 1}}^2.
\end{align*}
\end{lemma}
\begin{proof}
    Let $\tilde p_B^{(k)}$ denote the claimed simplification of $p_B^{(k)}$. Note that throughout this proof $h = \pi^{(k)}$ so that $h_i = \beta_{\nu(i+1)}$.

    We will prove $p_B^{(k)}=\tilde p_B^{(k)}$ inductively. For $k = 1$,
    we have that
    $\pi^{(1)} = [\sqrt{2}]$. Using the 
    relation $x_1 = x_0 - \sqrt{2} g_0$, we have that
    \begin{gather*}
        Q_{0,1}=f_0 - f_1 - \sqrt{2}\ip{g_1,g_0} -\frac{1}{2}\norm{g_0-g_1}^2,\qquad\text{and}\\
        Q_{1,0}=f_1 - f_0 + \sqrt{2}\norm{g_0}^2 - \frac{1}{2}\norm{g_0 - g_1}^2.
    \end{gather*}
    Thus,
    \begin{align*}
        p_B^{(1)} &= Q_{0,1} + \vr^{-1}Q_{1,0}\\
        &=\left[f_0 - f_1 - \sqrt{2}\ip{g_1,g_0} -\frac{1}{2}\norm{g_0-g_1}^2\right] + \vr^{-1}\left[f_1 - f_0 +\sqrt{2} \norm{g_0}^2 - \frac{1}{2}\norm{g_0 - g_1}^2\right]\\
        &= - \frac{\sqrt{2}}{\vr}\left(f_1 - f_0 - \frac{1}{2}\norm{g_0}^2\right) - \frac{1}{\vr}\norm{g_0}^2 - \frac{\sqrt{2}}{2}\norm{g_1}^2 = \tilde p_B^{(1)}.
    \end{align*}
    
    Now, inductively suppose $p_B^{(k)} = \tilde p_B^{(k)}$ and consider $p_B^{(k+1)}$.
    By definition, we have
    \begin{align*}
        B_{k+1} = \underbrace{\begin{pmatrix}
        B_{k}\\
        & \vr^2 B_{k}
        \end{pmatrix}}_{\eqqcolon\Theta} +  \underbrace{\begin{pmatrix}
            0 \\
            & 0 & \pi^{(k)} & 0 \\
             & & 0 & \\
               &  0 & \pi^{(k)} &  0
          \end{pmatrix}}_{\eqqcolon\Xi} + \underbrace{\begin{pmatrix}
            0 \\
            & 0 & 0 & 1 \\
             & & 0 & \\
               &  \vr^{k-1} & 0 &  0
          \end{pmatrix}}_{\eqqcolon \Delta}.
    \end{align*}
    


    Let $n = 2^k - 1$.
    The first and last $n$ entries of $\pi^{(k+1)}$ both coincide with $\pi^{(k)}$, thus by induction:
    \begin{align*}
        &\sum_{i,j} \Theta_{i,j}\bbracket{Q_{i,j}}_f  + \sum_{i,j} \Xi_{i,j}\bbracket{Q_{i,j}}_f\\
        &\qquad= \frac{1}{\vr}\sum_{i=0}^{n-1}h_i(f_i - f_{n}) + \vr \sum_{i=n+1}^{2n}h_i(f_i - f_{2n+1}) + \sum_{i=n+1}^{2n} h_i \left(f_{n} + f_{2n+1} - 2f_i\right)
        \\
        &\qquad= \bbracket{\tilde p_B^{(k+1)}}_f + \left((\vr^k -1)(1-\vr^{-1}) - \beta_k\vr^{-1}\right)f_{n} + \left((\vr^k -1)(1-\vr) + \vr^{-1}(\vr^{k+1} - 1)\right)f_{2n+1}.
    \end{align*}
    Here, we have used the fact that
    $\vr-2 = 1/\vr$ and $\sum_i \pi^{(k)}_i = \vr^k - 1$.
    Next,
    \begin{align*}
        \sum_{i,j}\Delta_{i,j}\bbracket{Q_{i,j}}_f &= \bbracket{Q_{n, 2n+1}}_f + \vr^{k-1}\bbracket{Q_{2n+1, n}}_f\\
        &= (\vr^{k-1} - 1)(f_{2n+1} - f_{n}).
    \end{align*}
    Adding up these two identities shows that
        $\bbracket{p_B^{(k+1)}}_f = \bbracket{\tilde p_B^{(k+1)}}_f$ (see {\color{red} \texttt{Mathematica Proof 3.1}}).

    We repeat this with $\bbracket{\cdot}_g$. Again, by induction:
    \begin{align*}
        &\sum_{i,j} \Theta_{i,j}\bbracket{Q_{i,j}}_g\\
        &\qquad= -\frac{1}{\vr} \sum_{i=0}^{n - 1} h_i \left(- \frac{1}{2}\norm{g_i}^2 - \ip{g_i, x_0 - x_i}\right) - \frac{1}{2\vr}\norm{x_n - x_0}^2 - \frac{\vr^{k-1}(\vr^k - 1)}{2}\norm{g_{n}}^2\\
        &\qquad\qquad - \vr \sum_{i=n+1}^{2n} h_i \left(- \frac{1}{2}\norm{g_i}^2 - \ip{g_i, x_{n+1} - x_i}\right) - \frac{\vr}{2}\norm{x_{2n+1} - x_{n+1}}^2 - \frac{\vr^{k+1}(\vr^k - 1)}{2}\norm{g_{2n + 1}}^2\\
        &\qquad= \bbracket{\tilde p_B^{(k+1)}} + 2 \sum_{i=n+1}^{2n} h_i \left(\frac{1}{2}\norm{g_i}^2 + \ip{g_i, x_0 - x_i}\right) + \frac{1}{2\vr}\norm{x_{2n+1} - x_0}^2 + \frac{\vr^k(\vr^{k+1}-1)}{2}\norm{g_{2n + 1}}^2
        \\
        &\qquad\qquad - \frac{1}{\vr}h_n \left(\frac{1}{2}\norm{g_n}^2 + \ip{g_n, x_0 - x_n}\right)
        - \frac{1}{2\vr}\norm{x_n - x_0}^2 - \frac{\vr^{k-1}(\vr^k - 1)}{2}\norm{g_{n}}^2\\
        &\qquad\qquad
        + \vr \ip{x_{n+1} - x_{2n + 1}, x_{n+1}-x_0}
        - \frac{\vr}{2}\norm{x_{2n+1} - x_{n+1}}^2 - \frac{\vr^{k+1}(\vr^k - 1)}{2}\norm{g_{2n + 1}}^2.
    \end{align*}
    Here, we have used the fact that $\rho - \rho^{-1} = 2$.
    Next, we compute:
    \begin{align*}
        &\sum_{i,j} \Xi_{i,j}\bbracket{Q_{i,j}}_g \\
        &\qquad=-\sum_{i=n+1}^{2n}h_i\left(\ip{g_i, x_{n} + x_{2n+1} - 2x_i} + \frac{1}{2}\norm{g_i - g_{n}}^2 + \frac{1}{2}\norm{g_i - g_{2n+1}}^2\right)\\
        &\qquad=-\sum_{i=n+1}^{2n}h_i\ip{g_i, x_{n} +  x_{2n+1} - 2x_0}  -2\sum_{i=n+1}^{2n}h_i\ip{g_i, x_0-x_i}\\
        &\qquad\qquad -\sum_{i=n+1}^{2n}h_i \norm{g_i}^2 -\frac{1}{2}\sum_{i=n+1}^{2n}h_i \norm{g_{n}}^2 -\frac{1}{2}\sum_{i=n+1}^{2n}h_i \norm{g_{2n+1}}^2\\
        &\qquad\qquad + \sum_{i=n+1}^{2n}h_i \ip{g_i, g_{n } + g_{2n+1}}\\
        &\qquad= 
        -2 \sum_{i=n+1}^{2n}h_i\left(\frac{1}{2}\norm{g_i}^2 + \ip{g_i, x_0 - x_i} \right) + 
        \ip{x_{2n+1} - x_{n+1}, x_n + x_{2n+1} - 2x_0} \\
        &\qquad\qquad - \frac{\vr^k - 1}{2}\norm{g_n}^2 - \frac{\vr^k - 1}{2}\norm{g_{2n+1}}^2 + \ip{x_{n+1} - x_{2n+1}, g_n + g_{2n+1}}.
    \end{align*}
    Here, we have used the facts that $\sum_{i=n+1}^{2n}h_i g_i =x_{n+1} - x_{2n+1}$ and that the last $n$ entries of $h$ are $\pi^{(k)}$.
    Next, 
    \begin{align*}
        \sum_{i,j}\Delta_{i,j}\bbracket{Q_{i,j}}_g
        &=\ip{g_{2n+1}, x_{2n+1}- x_{n}} - \frac{1}{2} \norm{g_{2n+1}- g_{n}}^2\\
        &\qquad - \vr^{k-1} \ip{g_{n}, x_{2n+1}- x_{n}} - \frac{\vr^{k-1}}{2} \norm{g_{2n+1}- g_{n}}^2\\
        &=\ip{ x_{2n+1} - x_n, g_{2n+1} - \rho^{k-1}g_{n}} - \frac{\beta_k}{2} \norm{g_{2n+1}- g_{n}}^2.
    \end{align*}

    We may now sum up these three identities to see that
    \begin{align*}
        \bbracket{p_B^{(k+1)}}_g &= \bbracket{\tilde p_B^{(k+1)}}_g + (\textup{a quadratic form in }x_n-x_0,\, x_{2n+1}-x_0,\, g_n,\, g_{2n+1}).
    \end{align*}
    This quadratic form is determined by a $4\by 4$ matrix which is identically zero, verified in {\color{red} \texttt{Mathematica Proof 3.2}}.
\end{proof}

\section{Proof of \Cref{thm:grad_norm_decrease} (An Analysis of $\hleft^{(k)}$)} \label{sec:h_left}
Define the following sequences of matrices $A^{(k)}\in \R^{2^k\times 2^k}$ inductively.
Set
\begin{gather*}
    A_1 \coloneqq \begin{pmatrix} 0 & 2 \\ 1 & 0\end{pmatrix}\qquad\text{and}\\
    A_{k+1} \coloneqq \begin{pmatrix}
        A_{k} &  \\
         & \frac{r_{k+1}}{\vr^{2k-1}} B_{k}
    \end{pmatrix}
    + 
    \frac{r_{k+1}}{2\vr^{2k}}
    \begin{pmatrix}
      0 \\
       & 0 & \pi^{(k)} & 1 
      \\
      & & 0 & \\
     &\vr^{k} - \frac{2\vr^{2k}}{r_{k+1}} &  \pi^{(k)}&  0
    \end{pmatrix}.
\end{gather*}
Here, in the second definition, the second matrix has support
\begin{align*}
    \set{(2^{k}-1,i):\, i\in [2^{k}-1, 2^{k+1} - 1]}
    \cup
    \set{(2^{k+1}-1, i):\, i\in [2^{k}-2, 2^{k+1}- 2]}.
\end{align*}


By \Cref{lem:rk_growth}, we have that $r_{k+1} \geq 4\vr^{k}$ so that $A^{(k)}$ is nonnegative for all $k\geq 1$.

Define
    $p_A^{(k)} \coloneqq \sum_{i,j=0}^{2^k-1} A^{(k)}_{i,j}Q_{i,j}$,
and let $\tilde p_A^{(k)}$ denote the claimed simplification of $p_A^{(k)}$ in \Cref{thm:grad_norm_decrease}. Note that throughout this proof $h= \hleft^{(k)}$.

We prove this inductively. For $k = 1$,
we have that
$\hleft^{(1)} = [3/2]$ and $r_1 = 4$. Using the 
relation $x_1 = x_0 - \frac{3}{2} g_0$, we have that
\begin{gather*}
    Q_{0,1}=f_0 - f_1 - \frac{3}{2}\ip{g_1,g_0} -\frac{1}{2}\norm{g_0-g_1}^2\geq 0\\
    Q_{1,0}=f_1 - f_0 + \frac{3}{2}\norm{g_0}^2 - \frac{1}{2}\norm{g_0 - g_1}^2 \geq 0.
\end{gather*}
Thus,
\begin{align*}
    p_A^{(1)} &= 2Q_{0,1} + Q_{1,0}\\
    &=2\left[f_0 - f_1 - \frac{3}{2}\ip{g_0, g_1} - \frac{1}{2}\norm{g_0 - g_1}^2\right] + \left[f_1 - f_0 +\frac{3}{2} \norm{g_0}^2 - \frac{1}{2}\norm{g_0 - g_1}^2\right]\\
    &= f_0 - f_1 - \frac{3}{2}\norm{g_1}^2 = \tilde p_A^{(1)}.
\end{align*}

Now, suppose $p_A^{(k)}=\tilde p_A^{(k)}$ and consider $p_A^{(k+1)}$. 
For notational simplicity, let $\gamma \coloneqq \frac{r_{k+1}}{\vr^{2k-1}}$.
By definition, we have
\begin{align*}
    A_{k+1}\coloneqq \underbrace{\begin{pmatrix}
    A_k&\\& \gamma B_k
    \end{pmatrix}}_{\eqqcolon\Theta} + \underbrace{\frac{\gamma}{2\rho}\begin{pmatrix}
    0 &\\& 0 & \pi^{(k)} & 0\\
    & & 0 & \\ 
    & 0 & \pi^{(k)} & 0
    \end{pmatrix}}_{\eqqcolon\Xi} + \underbrace{\begin{pmatrix}
        0 &\\& 0 & 0 & \frac{\gamma}{2\vr}\\
        & & 0 & \\ 
        & \frac{\gamma\vr^{k-1}}{2}- 1& 0 & 0
        \end{pmatrix}}_{\eqqcolon\Delta}
\end{align*}

Let $n=2^k - 1$.
Note that the first $n$ entries of $\hleft^{(k+1)}$ coincide with $\hleft^{(k)}$ and last $n$ entries of $\hleft^{(k+1)}$ coincide with $\pi^{(k)}$. Thus, by induction and \Cref{prop:p_b}, it holds that
\begin{gather*}
\begin{aligned}
    \sum_{i,j} \Theta_{i,j}\bbracket{Q_{i,j}}_f &= f_0 - f_n  + \frac{\gamma}{\vr}\sum_{i=n+1}^{2n} h_i(f_i - f_{2n + 1})\\
    &= f_0 - f_n  + \frac{\gamma}{\vr}\sum_{i=n+1}^{2n} h_if_i - \frac{\gamma(\vr^k - 1)}{\vr}f_{2n + 1},
\end{aligned}\\
\begin{aligned}
    \sum_{i,j} \Xi_{i,j}\bbracket{Q_{i,j}}_f &= \frac{\gamma}{2\vr}\left(\sum_{i=n+1}^{2n}h_i(f_n +f_{2n+1}- 2f_i)\right)\\
    &= \frac{\gamma (\vr^k - 1)}{2\vr}(f_n + f_{2n+1})- \frac{\gamma}{\vr}\sum_{i=n+1}^{2n}h_if_i,\quad\text{and}
\end{aligned}\\
\begin{aligned}
    \sum_{i,j} \Delta_{i,j}\bbracket{Q_{i,j}}_f &= \left(1 - \frac{\gamma (\vr^{k}-1)}{2\vr}\right)(f_n - f_{2n+1}).
\end{aligned}
\end{gather*}
Summing up these identities gives
\begin{align*}
    \bbracket{p_B^{(k+1)}}_f &= f_0 - f_{2n+1} = \bbracket{\tilde p_B^{(k+1)}}_f.
\end{align*}

We repeat this with $\bbracket{\cdot}_g$. First,
\begin{align*}
    &\sum_{i,j} \Theta_{i,j}\bbracket{Q_{i,j}}_g\\
    &\qquad= - \left(\frac{r_k - 1}{2}\right) \norm{g_{n}}^2 + \frac{\gamma}{2\vr}\sum_{i=n+1}^{2n} h_i \left(\norm{g_i}^2 + 2\ip{g_i, x_{n+1} - x_i}\right)\\
    &\qquad\qquad - \frac{\gamma}{2\vr}\norm{x_{2n+1} - x_{n+1}}^2 - \frac{\gamma\vr^{k-1}(\vr^k - 1)}{2}\norm{g_{2n+1}}^2\\
    &\qquad= \bbracket{\tilde p_A^{(k+1)}}+ \frac{\gamma}{2\vr}\sum_{i=n+1}^{2n} h_i \left(\norm{g_i}^2 +2\ip{g_i, x_0 - x_i}\right) - \left(\frac{r_k - 1}{2}\right) \norm{g_{n}}^2 +  \left(\frac{r_{k+1}-1}{2}\right) \norm{g_{2n+1}}^2 \\
    &\qquad\qquad  + \frac{\gamma}{\vr}\ip{x_{n+1}-x_{2n+1}, x_{n+1} - x_0} - \frac{\gamma}{2\vr}\norm{x_{2n+1} - x_{n+1}}^2 - \frac{\gamma\vr^{k-1}(\vr^k - 1)}{2}\norm{g_{2n+1}}^2.
\end{align*}
Next, we compute $\sum_{i,j} \Xi_{i,j}\bbracket{Q_{i,j}}_g$.
This quantity was previously computed (up to scaling) in the proof of \Cref{prop:p_b}. Note that in that proof, we only used the fact that the last $n$ entries of $h$ coincide with $\pi^{(k)}$. This also holds for $h = \hleft^{(k)}$. Thus, we have
\begin{align*}
    \sum_{i,j} \Xi_{i,j}\bbracket{Q_{i,j}}_g &= \frac{\gamma}{2\vr}\bigg(-2 \sum_{i=n+1}^{2n}h_i\left(\frac{1}{2}\norm{g_i}^2 + \ip{g_i, x_0 - x_i} \right) + 
    \ip{x_{2n+1} - x_{n+1}, x_n + x_{2n+1} - 2x_0} \\
    &\qquad\qquad\qquad - \frac{\vr^k - 1}{2}\norm{g_n}^2 - \frac{\vr^k - 1}{2}\norm{g_{2n+1}}^2 + \ip{x_{n+1} - x_{2n+1}, g_n + g_{2n+1}}\bigg).
\end{align*}
Finally, we have
\begin{align*}
    \sum_{i,j} \Delta_{i,j}\bbracket{Q_{i,j}}_g &= \frac{\gamma}{2\vr}\left(\ip{g_{2n+1}, x_{2n+1}- x_{n}} - \frac{1}{2} \norm{g_{2n+1}- g_{n}}^2\right)\\
    &\qquad\qquad - \left(\frac{\gamma \vr^k}{2\vr} - 1\right)\left(\ip{g_{n}, x_{2n+1}- x_{n}} + \frac{1}{2}\norm{g_{2n+1}- g_{n}}^2\right).
\end{align*}
Summing up the three identities above, shows that
\begin{align*}
    \bbracket{p_A^{(k+1)}}_g &=  \bbracket{\tilde p_A^{(k+1)}}_g + (\textup{a quadratic form in }x_n,\, x_{2n+1},\, g_n,\, g_{2n+1}).
\end{align*}
This quadratic form is determined by a $4\by 4$ matrix which is identically zero. This is verified in {\color{red} \texttt{Mathematica Proof 4.1}}, using the additional identities in Lemma~\ref{lem:r-formulas}.

\section{Proof of \Cref{thm:fn_val_decrease} (An Analysis of $\hright^{(k)}$)} \label{sec:h_right}
Define $c_k\in \R^{2^k}$ as follows
\begin{gather*}
    c_1 = \begin{pmatrix}
        1 & 1
    \end{pmatrix}\qquad\text{and}\\
    c_{k+1} = \begin{pmatrix} 
        \frac{1}{\sqrt{r_{k+1}}} \pi^{(k)} & \frac{\beta_{k+1}}{\sqrt{r_{k+1}}} & c_{k}
        \end{pmatrix}.
    \end{gather*}
    We will $0$-index $c_k$.
    Next, define $D_k\in\R^{2^k \by 2^k}$ as follows
    \begin{gather*}
    D_1 =\begin{pmatrix} 0 & 1/2 \\ 0 &  0\end{pmatrix}\qquad\text{and}\\
    D_{k+1} = 
\begin{pmatrix}
  \frac{\vr}{r_{k+1}} B_{k} & \\
  0 & D_{k}
\end{pmatrix} + \begin{pmatrix}
0 \\
& 0 & \left(\frac{1}{\sqrt{r_{k}}} - \frac{1}{\sqrt{r_{k+1}}}\right) c_{k}\\
& & 0
\end{pmatrix}.
\end{gather*}
Again, we will $0$-index both the rows and columns of $D_k$.
Here, the second matrix has support
\begin{gather*}
    \set{(2^{k} - 1, j):\, j\in[2^{k} , 2^{k+1} - 1]}.
\end{gather*}

Define,
    $p_D^{(k)}\coloneqq \sum_{i,j} (D_k)_{i,j}Q_{i,j}$.
and let $\tilde p_D^{(k)}$ denote the claimed simplification of $p_D^{(k)}$ in \Cref{thm:fn_val_decrease}. Note that throughout this proof $h = \hright^{(k)}$.

We will prove $p_D^{(k)} = \tilde p_D^{(k)}$ inductively. For $k = 1$, we have that $\hright^{(1)} = [3/2]$ and $r_1 = 4$. We compute
\begin{align*}
    p_D^{(1)} &= \frac{1}{2}Q_{0,1}\\
    &= \frac{1}{2}\left(f_0 - f_1 - \frac{3}{2}\ip{g_1, g_0} - \frac{1}{2}\norm{g_0 - g_1}^2\right)\\
    &= \frac{1}{2}\left(f_0 - f_1 + \frac{1}{2}\norm{g_0}^2\right) + \frac{1}{2}\left(\frac{1}{2}\norm{g_1}^2 + \frac{3}{2}\ip{g_1, g_0}\right) - \frac{1}{2}\norm{g_0 + g_1}^2 = \tilde p_D^{(1)}.
\end{align*}

Now, suppose $p_D^{(k)} = \tilde p_D^{(k)}$ and consider $p_D^{(k+1)}$. By definition, we have
\begin{align*}
    D_{k+1} = \underbrace{
\begin{pmatrix}
  \frac{\vr}{r_{k+1}} B_{k} & \\
  0 & D_{k}
\end{pmatrix}}_{\eqqcolon \Theta} + \underbrace{\begin{pmatrix}
0 \\
& 0 & \left(\frac{1}{\sqrt{r_k}} - \frac{1}{\sqrt{r_{k+1}}}\right) c_{k}\\
& & 0
\end{pmatrix}}_{\eqqcolon \Delta}.
\end{align*}
This process of recursively constructing such certificates was termed `recursive gluing' in \cite{altschuler2023accelerationPartI}.

Let $n = 2^k - 1$.
Note that the first $n$ elements of $h$ coincide with $\pi^{(k)}$ and the last $n$ elements of $h$ coincide with $\hright^{(k)}$. Thus, by \Cref{prop:p_b} and induction, we have that
\begin{align*}
    \sum_{i,j}\Theta_{i,j} \bbracket{Q_{i,j}}_f &= \frac{1}{r_{k+1}}\sum_{i=0}^{n-1}h_i (f_i
    -f_{n}) + \sum_{i=n+1}^{2n + 1} \frac{(c_k)_{i - 1- n} }{\sqrt{r_k}}\left(f_{i} - f_{2n+1}\right)\\
    &= \sum_{i=0}^{n-1}\frac{(c_{k+1})_i}{\sqrt{r_{k+1}}} (f_i
    -f_{n}) + \sum_{i=n+1}^{2n+1} \frac{(c_{k+1})_i }{\sqrt{r_k}}\left(f_{i} - f_{2n+1}\right)\\
    &= \bbracket{\tilde p_D^{(k+1)}} - \left(\frac{2\vr^k}{r_{k+1}}\right)f_n + \left(\frac{1}{\sqrt{r_k}}- \frac{1}{\sqrt{r_{k+1}}}\right)\sum_{i=n+1}^{2n+1}(c_{k+1})_i f_i.
\end{align*}
Here, we have used that the sum of $c_{k}$ is $\sqrt{r_k}$ (see Lemma~\ref{lem:sum_ck}) and that the last $2^k$ entries of $c_{k+1}$ are exactly the entries of $c_k$.
Next,
\begin{align*}
    \sum_{i,j}\Delta_{i,j} \bbracket{Q_{i,j}}_f&= 
    \left(\frac{1}{\sqrt{r_k}}- \frac{1}{\sqrt{r_{k+1}}}\right)\sum_{i=n+1}^{2n+1} (c_{k+1})_i\left(f_n - f_i\right)\\
    &= \left(1 - \frac{\sqrt{r_k}}{\sqrt{r_{k+1}}}\right)f_n  - \sum_{i=n+1}^{2n+1} \left(\frac{1}{\sqrt{r_k}}- \frac{1}{\sqrt{r_{k+1}}}\right)(c_{k+1})_i f_i.
\end{align*}

Summing up these identities shows
\begin{align*}
    \bbracket{p_D^{(k+1)}}_f &= \bbracket{\tilde p_D^{(k+1)}}_f  + \left(1 - \frac{\sqrt{r_k}}{\sqrt{r_{k+1} }} - \frac{2\vr^k}{r_{k+1}}\right)f_{n}.
\end{align*}
Note that the sum of all coefficient in the expressions $\bbracket{p_D^{(k+1)}}_f$ and $\bbracket{\tilde p_D^{(k+1)}}_f$ must be zero. Thus, the error term is identically zero.
Next, define $y_0,\dots, y_{2n+2}$ as $y_j\coloneqq x_0 - \sum_{i=0}^{j-1} (c_{k+1})_i g_i$.
Then,
\begin{align*}
    \sum_{i,j}\Theta_{i,j} \bbracket{Q_{i,j}}_g &= \frac{1}{r_{k+1}}\sum_{i=0}^{n-1} h_i\left(\frac{1}{2}\norm{g_i}^2 + \ip{g_i, x_0 - x_i}\right) - \frac{1}{2r_{k+1}}\norm{x_{n}- x_0}^2 - \frac{\vr^k(\vr^k - 1)}{2r_{k+1}}\norm{g_n}^2\\
    &\qquad + \sum_{i=n+1}^{2n +1} \frac{(c_{k+1})_i}{\sqrt{r_k}}\left(\frac{1}{2}\norm{g_i}^2 + \ip{g_i, x_{n+1} - x_i}\right) - \frac{1}{2}\norm{y_{n+1} - y_{2n+2}}^2\\
    &= \bbracket{\tilde p_D^{(k+1)}}_g + \frac{1}{2}\norm{y_{2n+2} - x_0}^2
    - \frac{\beta_{k+1}}{r_{k+1}}\left(\frac{1}{2}\norm{g_n}^2 + \ip{g_n, x_0-x_n}\right)\\
    &\qquad
    - \frac{1}{2r_{k+1}}\norm{x_{n}- x_0}^2  - \frac{\vr^k(\vr^k - 1)}{2r_{k+1}}\norm{g_n}^2\\
    &\qquad + \left(\frac{1}{\sqrt{r_k}} - \frac{1}{\sqrt{r_{k+1}}}\right)\sum_{i=n+1}^{2n +1} (c_{k+1})_i\left(\frac{1}{2}\norm{g_i}^2 + \ip{g_i, x_0 - x_i}\right) \\
    &\qquad + \frac{1}{\sqrt{r_k}}\ip{y_{n+1} - y_{2n+2}, x_{n+1} - x_0} - \frac{1}{2}\norm{y_{n+1}-y_{2n+2}}^2.
\end{align*}
And,
\begin{align*}
    \sum_{i,j}\Delta_{i,j} \bbracket{Q_{i,j}}_g &= \left(\frac{1}{\sqrt{r_k}}- \frac{1}{\sqrt{r_{k+1}}}\right)\sum_{i=n+1}^{2n+1} (c_{k+1})_i \left(- \ip{g_i, x_n - x_i}- \frac{1}{2}\norm{g_i - g_n}^2\right)\\
    &= \left(\frac{1}{\sqrt{r_k}}- \frac{1}{\sqrt{r_{k+1}}}\right)\sum_{i=n+1}^{2n+1} (c_{k+1})_i \left(- \ip{g_i, x_0 - x_i} - \frac{1}{2}\norm{g_i}^2\right) \\
    &\qquad + \left(\frac{1}{\sqrt{r_k}}- \frac{1}{\sqrt{r_{k+1}}}\right)\ip{y_{2n+2}-y_{n+1}, x_n - x_0 - g_n} -  \left(1- \frac{\sqrt{r_k}}{\sqrt{r_{k+1}}}\right)\frac{1}{2}\norm{g_n}^2.
\end{align*}

Summing up these two quantities gives
\begin{align*}
    \bbracket{p_D^{(k+1)}}_g &= \bbracket{\tilde p_D^{(k+1)}}_g + \frac{1}{2}\norm{y_{2n+2} - x_0}^2
    - \frac{\beta_{k+1}}{r_{k+1}}\left(\frac{1}{2}\norm{g_n}^2 + \ip{g_n, x_0-x_n}\right)\\
    &\qquad
    - \frac{1}{2r_{k+1}}\norm{x_{n}- x_0}^2  - \frac{\vr^k(\vr^k - 1)}{2r_{k+1}}\norm{g_n}^2\\
    &\qquad + \frac{1}{\sqrt{r_k}}\ip{y_{n+1} - y_{2n+2}, x_{n+1} - x_0} - \frac{1}{2}\norm{y_{n+1}-y_{2n+2}}^2\\
    &\qquad + \left(\frac{1}{\sqrt{r_k}}- \frac{1}{\sqrt{r_{k+1}}}\right)\ip{y_{2n+2}-y_{n+1}, x_n - x_0 - g_n} -  \left(1- \frac{\sqrt{r_k}}{\sqrt{r_{k+1}}}\right)\frac{1}{2}\norm{g_n}^2.
\end{align*}
Note that $y_0 = x_0$ and that $y_{n} - y_0 = -\sum_{i=0}^{n-1} (c_{k+1})_i g_i = \frac{1}{\sqrt{r_{k+1}}}(x_{n}- x_0)$.
Thus, the error term is a quadratic form in the quantities: $x_n- x_0$,
$x_{2n+1} - x_0$,
$g_n$,
$g_{2n +1}$,
$y_{2n+2} - y_{n+1}$. The $5\by 5$ matrix determining this quadratic form is verified to be identically zero in {\color{red} \texttt{Mathematica Proof 5.1}}.

\paragraph{Acknowledgements.} Benjamin Grimmer's work was supported in part by the Air Force Office of Scientific Research under award number FA9550-23-1-0531.

\bibliographystyle{unsrt}
\bibliography{bibliography}

\begin{thebibliography}{10}

\bibitem{TechnicalLongSteps}
Benjamin Grimmer, Kevin Shu, and Alex~L. Wang.
\newblock Accelerated gradient descent via long steps, 2023.

\bibitem{drori2012PerformanceOF}
Yoel Drori and Marc Teboulle.
\newblock Performance of first-order methods for smooth convex minimization: a novel approach.
\newblock {\em Mathematical Programming}, 145:451--482, 2012.

\bibitem{taylor2017interpolation}
Adrien Taylor, Julien Hendrickx, and Fran\c{c}ois Glineur.
\newblock Smooth strongly convex interpolation and exact worst-case performance of first-order methods.
\newblock {\em Mathematical Programming}, 161:307–345, 2017.

\bibitem{taylor2017smooth}
Adrien Taylor, Julien Hendrickx, and Fran\c{c}ois Glineur.
\newblock Exact worst-case performance of first-order methods for composite convex optimization.
\newblock {\em SIAM Journal on Optimization}, 27(3):1283--1313, 2017.

\bibitem{Teboulle2022}
Marc Teboulle and Yakov Vaisbourd.
\newblock An elementary approach to tight worst case complexity analysis of gradient based methods.
\newblock {\em Math. Program.}, 201(1–2):63–96, oct 2022.

\bibitem{Daccache2019}
Antoine Daccache.
\newblock Performance estimation of the gradient method with fixed arbitrary step sizes.
\newblock Master's thesis, Universit\'e Catholique de Louvain, 2019.

\bibitem{Eloi2022}
Diego Eloi.
\newblock Worst-case functions for the gradient method with fixed variable step sizes.
\newblock Master's thesis, Universit\'e Catholique de Louvain, 2022.

\bibitem{gupta2023branch}
Shuvomoy~Das Gupta, Bart P.G.~Van Parys, and Ernest Ryu.
\newblock Branch-and-bound performance estimation programming: A unified methodology for constructing optimal optimization methods.
\newblock {\em Mathematical Programming}, 2023.

\bibitem{altschuler2018greed}
Jason Altschuler.
\newblock Greed, hedging, and acceleration in convex optimization.
\newblock Master's thesis, Massachusetts Institute of Technology, 2018.

\bibitem{Grimmer2023-long}
B.~Grimmer.
\newblock {Provably Faster Gradient Descent via Long Steps}.
\newblock {\em arxiv:2307.06324}, 2023.

\bibitem{altschuler2023accelerationPartI}
Jason~M. Altschuler and Pablo~A. Parrilo.
\newblock Acceleration by stepsize hedging i: Multi-step descent and the silver stepsize schedule, 2023.

\bibitem{altschuler2023accelerationPartII}
Jason~M. Altschuler and Pablo~A. Parrilo.
\newblock Acceleration by stepsize hedging ii: Silver stepsize schedule for smooth convex optimization, 2023.

\bibitem{kim2023timereversed}
Jaeyeon Kim, Asuman Ozdaglar, Chanwoo Park, and Ernest~K. Ryu.
\newblock Time-reversed dissipation induces duality between minimizing gradient norm and function value, 2023.

\bibitem{Nesterov1983}
Yurii Nesterov.
\newblock A method for solving the convex programming problem with convergence rate $o(1/k^2)$.
\newblock {\em Proceedings of the USSR Academy of Sciences}, 269:543--547, 1983.

\bibitem{Yudin1982}
A.~Nemirovski and D.~Yudin.
\newblock {\em Problem Complexity and Method Efficiency in Optimization}.
\newblock Wiley-Interscience, 1983.

\bibitem{KimFessler2021}
Donghwan Kim and Jeffrey~A. Fessler.
\newblock Optimizing the efficiency of first-order methods for decreasing the gradient of smooth convex functions.
\newblock {\em J. Optim. Theory Appl.}, 188(1):192–219, jan 2021.

\bibitem{luner2024averaging}
Alan Luner and Benjamin Grimmer.
\newblock On averaging and extrapolation for gradient descent, 2024.

\bibitem{Taylor2015SmoothSC}
Adrien~B. Taylor, Julien~M. Hendrickx, and François Glineur.
\newblock Smooth strongly convex interpolation and exact worst-case performance of first-order methods.
\newblock {\em Mathematical Programming}, 161:307 -- 345, 2015.

\bibitem{nesterov-2018-textbook}
Yurii Nesterov.
\newblock {\em Lectures on Convex Optimization}.
\newblock Springer Publishing Company, Incorporated, 2nd edition, 2018.

\bibitem{mathematica}
Wolfram~Research{,} Inc.
\newblock Mathematica, {V}ersion 13.3.
\newblock Champaign, IL, 2023.

\end{thebibliography}

\appendix
\section{Useful Identities}
\label{sec:useful_identitites}
The following formulas relate to the silver stepsize schedule $\pi^{(k)}$.
\begin{lemma}
	\label{lem:silver_identities}
	For $k\geq 1$, we have that
	\begin{gather*}
		\sum_{i=0}^{2^k - 2}\pi^{(k)}_i = \vr^k - 1,\qquad\text{and}\qquad
		\prod_{i=0}^{2^k - 2} \left(\pi^{(k)}_i - 1\right)^2 = \vr^{-2k}.
	\end{gather*}
\end{lemma}
\begin{proof}
	For $k = 1$, we have that $\pi^{(1)} = [\sqrt{2}]$. 
	The first identity is easily verified. The second identity can be seen by noting that $\sqrt{2} - 1 = \vr^{-1}$.
	For $k>1$, both of these results follow by induction, noting that $1 + 2\vr = \vr^2$,
	\begin{align*}
		\sum_{i=0}^{2^{k+1} - 2}\pi^{(k+1)}_i 
		&= (1 + 2\vr)\vr^{k-1}-1 = \vr^{k+1} - 1.\\
		\prod_{i=0}^{2^{k+1} - 2} (\pi^{(k+1)}_i - 1)^2 &= \vr^{-4k}\vr^{2(k-1)} = \vr^{-2(k+1)}.\qedhere
	\end{align*}
\end{proof}

The remainder of this appendix presents useful identities and inequalities regarding $\alpha_k,\,\beta_k,\,r_k$.
\begin{lemma}
	\label{lem:explicit_alpha_mu}
	Let $k \geq 1$. It holds that $\alpha_k$ is the unique root larger than $1$ of
	\begin{align*}
		x\mapsto r_k(x - 1 - \vr^k) + 2(x - 1)^2.
	\end{align*}
	This gives the explicit formulas:
	\begin{gather*}
		\alpha_{k} = 1 + \frac{\sqrt{r_k(r_k + 8 \vr^k)}-r_k}{4},\qquad\text{and}\\
		r_{k+1} = \frac{r_k+4\vr^k+\sqrt{r_k(r_k+ 8\vr^k)}}{2}.
	\end{gather*}
\end{lemma}
\begin{proof}
	Fix $k\geq 1$. We will rewrite the second constraint in \eqref{eq:defining_alpha}.
	Recall that $\hleft^{(k+1)} = [\hleft^{(k)},\alpha_k, \pi^{(k)}]$. 
	We will use \Cref{lem:silver_identities} and the definition of $r_k$ to simplify the partial sums and products in \eqref{eq:defining_alpha}.
	Thus, $\alpha_k$ is the unique root larger than $1$ of
	\begin{align}
		\label{eq:alpha_simplified_equation}
		\frac{1}{r_k + 2(\vr^k - 1) + 2x} - \frac{(x-1)^2}{\vr^{2k}r_k} = 0.
	\end{align}
	The quadratic in the lemma statement is derived from this formula by multiplying the denominators in \eqref{eq:alpha_simplified_equation} through to get a cubic equation in $x$. Then, we may verify that $x = 1-\vr^k$ is a root (smaller than 1) of the resulting cubic so that we may factor out this root to leave  a quadratic expression in $x$ with a unique root larger than $1$.
	
	The formula for $r_{k+1}$ follows from plugging in our formula for $\alpha_k$ into the identity
	\begin{gather*}
		r_{k+1} = r_k + 2(\vr^k -1) + 2\alpha_k.\qedhere
	\end{gather*}
\end{proof}
\begin{lemma}
	\label{lem:rk_growth}
	Let $k \geq 1$, then
	\begin{align*}
		r_k \simeq \left((\vr-1)\left(1 + \vr^{-1/2}\right)\right)\vr^k.
	\end{align*}
	Furthermore, $r_k \geq 4\vr^{k-1}$.
\end{lemma}
\begin{proof}
	We will consider the sequence $\gamma_k\coloneqq r_k \vr^{-k}$. By \Cref{lem:explicit_alpha_mu}, we have that
	\begin{align*}
		\gamma_{k+1} &\coloneqq \frac{r_{k+1}}{\vr^{k+1}}
		= \frac{r_k+4\vr^k+\sqrt{r_k(r_k+ 8\vr^k)}}{2\vr^{k+1}}
		= \frac{4 + \gamma_k + \sqrt{\gamma_k(\gamma_k + 8)}}{2\vr}.
	\end{align*}
	Thus, after defining
	\begin{align*}
		f(x) \coloneqq \frac{4 + x + \sqrt{x(x + 8)}}{2\vr},
	\end{align*}
	we may write the recurrence for $\gamma_k$ as
	$\gamma_1 = \frac{4}{\vr}$ and 
	$\gamma_{k+1} = f(\gamma_k)$.
	
	Now, note that $f$ is an increasing concave function satisfying $f(0)>0$ and that $\lim_{x\to\infty}f'(x) < 1$.
	Thus, there is a unique positive solution $\gamma^*$ to the equation $f(x) = x$.
	One may check that $\gamma^*$ is given explicitly by $\gamma^* = (\vr-1)(1 + \sqrt{\vr^{-1}})$.
	
	Our next goal is to show that $\gamma_k$ is an increasing sequence contained in $[0,\gamma^*)$.
	First, note that $\gamma_0 = 4/\vr <\gamma^*$.
	Inductively, suppose $\gamma_k\in [0,\gamma^*)$. As
	$f(x) - x >0$ on all of $[0,\gamma^*)$, we deduce that $\gamma_{k+1} = f(\gamma_k) >\gamma_k$.
	Next, we claim that $\gamma_{k+1}<\gamma^*$. Supposing otherwise, we would have that $f(\gamma_k)\geq \gamma^* = f(\gamma^*) > f(\gamma_k)$, where the last step follows as $f$ is an increasing function.
	Thus, $\gamma_{k+1} \in(\gamma_k, \gamma^*)$.
	
	Finally, we claim that $\lim_{k\to\infty}\gamma_k = \gamma^*$. This limit exists as $\gamma_k$ is an increasing bounded sequence. Call the limit $\gamma'$. Then
	$f(\gamma') = \lim_{k\to\infty} f(\gamma_k) = \lim_{k\to\infty} \gamma_k = \gamma'$. As $\gamma^*$ is the unique nonnegative solution to $f(\gamma) = \gamma$, we conclude that $\gamma' = \gamma^*$.
\end{proof}

\begin{lemma}
	For $k\geq 0$, it holds that $\beta_k < \alpha_k < \beta_{k+1}$.
\end{lemma}
\begin{proof}
    Fix $k\geq 0$ and let $q(x)\coloneqq r_k(x - 1 - \vr^k) + 2(x-1)^2$.
	By \Cref{lem:explicit_alpha_mu}, it holds that $\alpha_k$ is the unique root of this quadratic larger than $1$.
	Now, note that $q(1)<0$ and $q(\beta_{k+1})>0$. We conclude that $\alpha_k < \beta_{k+1}$.
	We will use the same strategy to show that $\alpha_k\geq \beta_k$. Specifically, note that
	\begin{align*}
		q(\beta_k) &= r_k(\beta_k - \beta_{k+1}) + 2(\beta_k-1)^2\\
		&= \vr^{k-1}(1-\vr)r_k + 2\vr^{2(k-1)}\\
		&\leq 4\vr^{2(k-1)}(1-\vr) + 2\vr^{2(k-1)}\\
		&= \left(4(1-\vr) + 2\right)\vr^{2(k-1)}
		<0.
	\end{align*}
	Here, the first inequality follows from \Cref{lem:rk_growth}.
\end{proof}

\begin{lemma}
	\label{lem:r-formulas}
	For all $k\geq0$, it holds that
	\begin{gather*}
		r_k = 2\frac{(\alpha_{k}-1)^2}{\beta_{k+1}-\alpha_{k}}, \qquad \text{and}\qquad
		r_{k+1} = 2\frac{(\beta_{k+1}-1)^2}{\beta_{k+1}-\alpha_k}.
	\end{gather*}
	In particular, $\frac{\sqrt{r_k}}{\alpha_{k}-1} =  \frac{\sqrt{r_{k+1}}}{\beta_{k+1}-1}$.
\end{lemma}
\begin{proof}
	Let $q(x)\coloneqq r_k(x - 1 - \vr^k) + 2(x-1)^2$.
	By \Cref{lem:explicit_alpha_mu}, it holds that $\alpha_k$ is the unique root of this quadratic larger than $1$. Thus,
	\begin{align*}
		2(\alpha_{k}-1)^2 - r_k(\beta_{k+1} - \alpha_{k}) = 0.
	\end{align*}
	Solving this equation for $r_k$ shows that
	$
	r_k = 2\frac{(\alpha_{k}-1)^2}{\beta_{k+1}-\alpha_{k}}.
	$
	Recall that
    $r_{k+1} = 1 + 2\sum_{i=0}^{2^{k+1} - 2} (\hright^{(k+1)})_i$. Thus, by \Cref{lem:silver_identities}, we have that    
    $r_{k+1} = r_k + 2((\alpha_{k}-1) + (\beta_{k+1}-1)) $, so
	\begin{align*}
		r_{k+1} &= 2\frac{(\alpha_{k}-1)^2}{(\beta_{k+1}-1)-(\alpha_{k}-1)} + 2((\alpha_{k}-1) + (\beta_{k+1}-1))\\
		&= 2\frac{(\alpha_{k}-1)^2 + \left((\beta_{k+1}-1)^2 - (\alpha_{k}-1)^2\right)}{(\beta_{k+1}-1)-(\alpha_{k}-1)}\\
		&= 2\frac{(\beta_{k+1}-1)^2}{\beta_{k+1}-\alpha_{k}}.
	\end{align*}
	Setting these two formulas for $r_k$ equal 
	and taking the square root implies the last claim.
\end{proof}

\begin{lemma} \label{lem:sum_ck}
	For any $k\geq 1$, $\sum_{i=0}^{2^k-1} (c_k)_i = \sqrt{r_k}.$
\end{lemma}
\begin{proof}
	For $k=1$, $\sum_{i=0}^{1} (c_1)_i = (c_1)_0 + (c_1)_1 = 2$ while $r_1 = 2(\hright^{(1)})_0 + 1 = 4$. For $k>1$, this follows inductively as
	\begin{align*}
		\sum_{i=0}^{2^k-1} (c_k)_i &= \frac{1}{\sqrt{r_{k}}} \sum_{i=0}^{2^{k-1} - 2} (\pi^{(k-1)})_i + \frac{\beta_k}{\sqrt{r_{k}}} + \sum_{i=0}^{2^{k-1}-1} (c_{k-1})_i\\
		&= \frac{2\vr^{k-1}}{\sqrt{r_k}} + \sqrt{r_{k-1}}\\
		&= \sqrt{r_{k}}		
	\end{align*}
	where the first equality expands the recursive definition of $c_k$, the second uses Lemma~\ref{lem:silver_identities} and the inductive hypothesis, and the third equality is equivalent to the identity
	$ r_k r_{k-1} = (r_k -2 \vr^{k-1})^2 $
	which is implied by the formula for $r_k$ in Lemma~\ref{lem:explicit_alpha_mu}.
\end{proof}

\end{document}